\documentclass[final]{siamltex}

\usepackage{amsfonts, amsmath, amssymb} 

\usepackage{graphicx}   

\newtheorem{remark}{Remark}

\title{\sc Block Triangular Preconditioning for Stochastic Galerkin Method}
\author{Bin Zheng, Guang Lin, and Jinchao Xu}
\date{November 9, 2012}

\begin{document}

\maketitle

\begin{abstract}
In this paper we study fast iterative solvers for the large sparse linear 
systems resulting from the stochastic Galerkin discretization of stochastic partial differential 
equations. A block triangular preconditioner is introduced and applied to the Krylov
subspace methods, including the generalized minimum residual method and the generalized preconditioned 
conjugate gradient method. This preconditioner utilizes the special 
structures of the stochastic Galerkin matrices to achieve high efficiency. Spectral 
bounds for the preconditioned matrix are provided for convergence analysis.
The preconditioner system can be solved approximately by geometric multigrid V-cycle. 
Numerical results indicate that the block triangular preconditioner has 
better performance than the traditional block diagonal preconditioner for stochastic problems 
with large variance. 

\end{abstract}

\maketitle


\section{Introduction}

Stochastic partial differential equations (SPDEs) provide mathematical models for uncertainty quantification in many physical and engineering applications, including flows in heterogeneous porous media \cite{ma-zabaras-2011-mixed}, thermo-fluid processes \cite{maitre-knio-najm-ghanem-2001, maitre-knio-najm-ghanem-2002}, flow-structure interactions \cite{xiu-lucor-su-karniadakis}, etc. These models are proposed to quantify the propagation of uncertainties in the input data, such as coefficients, forcing terms, boundary conditions, initial conditions, and geometry of the domain, to the response quantities of interests. 
%

Numerical methods for solving SPDEs can be categorized as either non-intrusive or intrusive methods. Non-intrusive methods are based on sampling techniques where a number of uncoupled deterministic partial differential equations need to be solved. Examples of such methods are Monte Carlo method and stochastic collocation method. Monte Carlo method is the most straightforward approach and its convergence is independent of the number of stochastic dimensions. It is useful for problems with very large stochastic dimensions where most of the other methods suffer from the ``curse of dimensionality". The disadvantage of Monte Carlo method is that it has a very slow rate of convergence, i.e., proportional to $1/\sqrt{N}$ where $N$ is the number of samples. Stochastic collocation method combines a Galerkin approximation in the physical space and a collocation/interpolation in the zeros of suitable orthogonal polynomials in the probability space. It achieves fast convergence for problems with moderate stochastic dimensions and smooth solutions in the stochastic domain \cite{tatang-pan-prinn-mcrae, mathelin-hussaini-zan, babuska-nobile-tempone}. However, the number of collocation points in a tensor grid grows exponentially with respect to the number of random variables. To overcome this problem, the Smolyak sparse grid stochastic collocation method is developed for problems with high random dimensions \cite{xiu-hesthaven}. The main advantage of the aforementioned non-intrusive methods is the ease of implementation, namely, existing deterministic solvers can be employed without any modification to solve the uncoupled deterministic problems. 

Stochastic Galerkin method  \cite{ghanem-spanos} is considered as an intrusive method in the sense that it results in coupled systems which cannot be solved by deterministic solvers directly. One advantage of the stochastic Galerkin method is that the number of equations is relatively small and scales as $\approx 1/2^p$ ($p$ is the order of the stochastic discretization) times the number of sparse grid stochastic collocation equations \cite{xiu-hesthaven, xiu-shen}, which is appealing for high order stochastic discretization. In \cite{elman-miller-phipps-tuminaro, back-nobile-tamellini-tempone}, experimental comparisons of the stochastic Galerkin method and stochastic collocation method are provided which show that the stochastic Galerkin method is advantageous in terms of computational cost when efficient solvers are available. Stochastic Galerkin method is based on the concept of generalized polynomial chaos expansion \cite{wiener-1938, xiu-karniadakis} which provides an exponentially convergent approximation for modeling smooth stochastic processes. It applies Galerkin projection onto a space of generalized polynomial chaos in the probability space. Such stochastic discretization is often combined with a finite element discretization in the physical space. The resulting method is often called the stochastic Galerkin finite element method for SPDEs. However, the size of the coupled linear system grows dramatically when increasing the solution resolution in both stochastic space and physical space. Hence, the development of fast solvers is necessary for solving these coupled linear systems efficiently. Fortunately, the matrix of this coupled linear system has particular structures which can be utilized in designing fast solvers. In particular, although deterministic solvers can not be applied directly to the coupled linear system, they may be used as building blocks for the design of new solvers.

There are a number of studies on fast solvers for the large coupled linear system including preconditioned Krylov subspace methods and multigrid methods. We refer to \cite{rosseel-vandewalle, tipireddy-phipps-ghanem} for comprehensive overviews and comparisons of iterative solvers for stochastic Galerkin discretizations. Due to the sparsity pattern of the linear system, Krylov subspace methods that require only matrix-vector multiplication are very attractive. Preconditioning techniques have been studied to accelerate the convergence of the Krylov subspace methods. The block diagonal preconditioner (also known as mean-based preconditioner) for the conjugate gradient (CG) method is the most convenient one that has been observed to be robust for problems with small variance \cite{ghanem_kruger, pellissetti-ghanem}. In \cite{powell-elman}, theoretical eigenvalue bounds for the block diagonal preconditioned system matrix are derived. Also, in \cite{ullmann}, a symmetric positive definite Kronecker product preconditioner has been introduced which makes use of the entire information contained in the coupled linear system to achieve better performance in terms of the CG iteration counts. Multigrid methods with optimal order of computational complexity in the physical space are also investigated theoretically and numerically for stochastic Galerkin discretizations \cite{maitre-knio-debusschere-najm-ghanem, seynaeve-rosseel-nicolai-vandewalle, elman-furnival-mg, rosseel-boonen-vanderwalle}. The hierarchical structure in the stochastic dimension resulting from the use of hierarchical polynomial chaos basis functions is explored in the design of iterative solvers \cite{ghanem_kruger, pellissetti-ghanem, rosseel-vandewalle}.

In this work, we combine the optimality of the multigrid method in physical space with the hierarchical structure in probability space to obtain efficient preconditioner for the Krylov type iterative methods. In particular, we propose a block triangular preconditioner for the generalized minimal residual (GMRes) method, and the generalized preconditioned conjugate gradient (GPCG) method. We also consider a symmetric block preconditioner for the standard conjugate gradient (CG) method. Numerical results indicate that block triangular preconditioner is more efficient than the traditional block-diagonal preconditioner for problems with large random fluctuations. 


The rest of the paper is organized as follows. We start in Section \ref{sec: mode-discretization} from a description of a model elliptic problem with random diffusion coefficient and an overview of the corresponding stochastic Galerkin finite element discretization. In Section \ref{sec: matrix structures}, we summarize the structures of the stochastic Galerkin matrix. In Section \ref{sec: block preconditioners}, we introduce the block triangular preconditioner and its symmetrized version.  Convergence analysis of the preconditioned Krylov subspace methods is described in Section \ref{sec: convergence analysis}. Finally, the performance of the proposed block triangular preconditioner for GMRes and GPCG method are demonstrated in Section \ref{sec: numerical results}.

\section{Model problem and stochastic Galerkin discretization}
\label{sec: mode-discretization}

We consider the following elliptic problem
\begin{eqnarray}
-\nabla\cdot (a(x,\omega) \nabla u(x,\omega)) &=& f(x),\;x\in D\subset \mathbb{R}^d,
\label{model-problem} 
\\ \nonumber
u(x, \omega) &=& 0, \;x\in \partial D,
\end{eqnarray}
where the diffusion coefficient $a$ is a real-valued random field defined on $D$, i.e. for each $x\in D$, $a(x,\cdot)$ is a random variable with respect to the probability space $(\Omega, \mathcal{F}, P)$. We assume that $a$ is bounded and uniformly coercive, i.e.
\begin{equation}
\exists\; a_{\min}, a_{\max}\in (0,+\infty):\;P(\omega\in\Omega:\;a(x,\omega)\in [a_{\min},a_{\max}],\;\forall x\in \overline{D})=1.
\label{assumption}
\end{equation}

Introducing the tensor product Hilbert space $V=L^2_P(\Omega) \otimes H_0^1(D)$ with inner product defined by
\begin{equation*}
(u, v)_V =
\int_\Omega\left(
\int_D \nabla u(x, \omega) \cdot \nabla v(x, \omega)\,\mathrm{d}x \right)
\,\mathrm{d} P(\omega),
\end{equation*} 
the weak solution $u \in V$ is a random function such that 
$\forall \;v\in V$:
\begin{equation}
\int_\Omega\left(
\int_D a(x, \omega) \nabla u(x, \omega) \cdot \nabla v(x, \omega)\,\mathrm{d}x \right)
\,\mathrm{d} P(\omega)
= \int_\Omega \left(
\int_D f(x) \,v(x, \omega)\,\mathrm{d}x
\right)\,\mathrm{d} P(\omega).
\label{weak-form}
\end{equation}
The well-posedness of the above variational problem (\ref{weak-form}) follows from (\ref{assumption}) and the Lax-Milgram lemma.

\subsection{Karhunen-Lo\`{e}ve expansion}
The first step of the stochastic discretization is to approximate the input random field $a(x, \omega)$ by the truncated Karhunen-Lo\`{e}ve expansion
\begin{equation}
a(x, \omega)\approx a_m(x, \omega) := \bar{a}(x) + \sum_{k=1}^m \sqrt{\lambda_k} b_k(x) \xi_k(\omega),
\label{truncated-KLE}
\end{equation}
where $\bar{a}(x)$ is the mean value of $a(x, \omega)$, $\lambda_k$ and $b_k(x)$ are the eigenvalues and eigenfunctions of the integral operator $C: L^2(D)\rightarrow L^2(D)$ defined by
$
\int_D \mathrm{Cov}_a (x, \cdot) \, u(x)\, \mathrm{d}x.
$
Given mean value $\bar{a}(x)$ and a continuous covariance function $\mathrm{Cov}_a(x, y)$, the truncated Karhunen-Lo\`{e}ve expansion (\ref{truncated-KLE}) approximates $a(x, \omega)$ with minimized mean square error \cite{ghanem-spanos}. The number of terms in the expansion, $m$, is determined by the eigenvalue decay rate, and in turn, depends on the stochastic regularity, i.e., the smoothness of the covariance function. The expansion coefficients $\xi_k(\omega)$ are pairwise uncorrelated random variables with images $\Gamma_k = \xi_k(\Omega)$, and probability density functions (PDFs) $\rho_k: \Gamma_k \rightarrow \mathbb{R}^n$. The joint PDF of the random vector $\xi=(\xi_1, \dots, \xi_m)$ is denoted by $\rho(\xi)$, and the image $\Gamma = \Pi_{k=1}^m \Gamma_k$. 

If $a(x, \omega)$ is a Gaussian random field, $\xi_k$ will be independent Gaussian random variables, and the joint PDF $\rho(\xi)=\Pi_{k=1}^m \rho_k(\xi_k)$. In general, for non-Gaussian random field, $\xi_k$ are not necessarily independent and their distributions are not known. Several methods have been developed to estimate the distributions of $\xi_k$ and to simulate non-Gaussian processes using Karhunen-Lo\`{e}ve expansion, see \cite{phoon-huang-quek, wan-karniadakis}. A non-Gaussian random field may also be approximated by polynomial chaos expansion, see \cite{ghanem-non-gaussian-1999, puig-poirion-soize}.
An example of the covariance function is given by the exponential covariance function
\begin{equation}
\mathrm{Cov}_a (x, y) = \sigma^2 \mathrm{exp}\left(-|x-y|/L\right),
\label{exponential-covariance}
\end{equation}
where $\sigma$ is the standard deviation and $L$ is the correlation length.

\begin{remark}
When replacing the diffusion coefficient $a(x, \omega)$ by the truncated Karhunen-Lo\`{e}ve expansion $a_m(x, \omega)$, it is important to verify the uniform coercivity condition (\ref{assumption}) so that the problem
\begin{eqnarray}
-\nabla\cdot (a_m(x,\omega) \nabla u(x,\omega)) &=& f(x),\;x\in D,
\label{approximated-model}
\\ \nonumber
u(x, \omega) &=& 0, \;x\in \partial D.
\end{eqnarray}
 is well-posed. For more discussions, including the estimate of the error between the two solutions of (\ref{model-problem}) and (\ref{approximated-model}), we refer to \cite{frauenfelder-schwab-todor, babuska-chatzipantelidis}.
\end{remark}

One advantage of using Karhunen-Lo\`{e}ve expansion is the separation of the stochastic and deterministic variables for the stochastic function $a(x, \omega)$. In addition, by the Doob-Dynkin lemma \cite{oksendal}, the solution of (\ref{approximated-model}) can be written in terms of $\xi$, i.e. $u(x, \omega) = u(x, \xi_1(\omega), \dots, \xi_m (\omega))$. The stochastic problem (\ref{approximated-model}) is then reformulated as the following deterministic parametrized problem:
\begin{eqnarray}
-\nabla\cdot (a_m(x, \xi) \nabla u(x,\xi)) &=& f(x),\;x\in D,\;\xi\in\Gamma,
\label{approximated-model-deterministic}
\\ \nonumber
u(x, \xi) &=& 0, \;x\in \partial D,\;\xi\in \Gamma.
\end{eqnarray}

\subsection{Stochastic Galerkin discretization}
Since the weak solution $u(x, \xi)$ is defined in a tensor product space $V$, we consider finite dimensional approximation space also in tensor product form, i.e. $V_{h, p}=X_h\otimes \Xi_{p}$. When the solution is smooth/analytic in stochastic variables, spectral approximation using global polynomials of total degree $\leq p$ in $m$ variables defined in $\Gamma$
$$
\Xi_{p}=\mathrm{span}\{\psi_1(\xi), \dots, \psi_{N_\xi}(\xi)\}\subset L_{\rho}^2(\Gamma)
$$ 
are good candidates for approximations in the stochastic space. The global polynomials $\psi_i$ are chosen to be orthogonal polynomials associated with the density function $\rho$, often referred to as generalized Polynomial Chaos (gPC) \cite{xiu-karniadakis}, e.g., Legendre polynomials for uniform distribution, Hermite polynomial for Gaussian distribution, etc. The dimension of the space $\Xi_p$ is given by the following formula
\begin{equation*}
N_\xi = \frac{(m+p)!}{m!p!}.
\end{equation*}
For example, when $m=6$, $p=4$, $\mathrm{dim}(\Xi_{4}) = 210$.

For the spatial approximation, we choose the standard finite element space, i.e., space of piecewise polynomials with respect to a given mesh $D_h$ ($h$ is the spatial discretization parameter)
$$
X_h=\mathrm{span}\{\phi_1(x), \dots, \phi_{N_x}(x)\}\subset H^1_0(D),
$$
where $N_x$ is the dimension of $X_h$. Hence, the discrete solution $u_{h, p}$ can be written as the following polynomial chaos expansion
\begin{equation}
u_{h, p}(x, \xi) = \sum_{j=1}^{N_\xi} u_j(x) \psi_j(\xi)
=\sum_{j=1}^{N_\xi} \left(\sum_{s=1}^{N_x} U_{j, s} \phi_s(x) \right)\psi_j(\xi).
\label{PCE_u}
\end{equation}
An a priori error estimate for the error $\|u-u_{h, p}\|_{L^2_\rho(\Gamma)\otimes H_0^1(D)}$ is given in \cite{babuska-nobile-tempone}. Statistical information including mean, variance, etc., can then be obtained from the explicit formula given by Eqn. (\ref{PCE_u}), which give good approximations to those of the exact solution.

The stochastic Galerkin finite element method is obtained by applying the Galerkin projection in the tensor product space $V_{h, p}$. More precisely, find $u_{h, p}\in V_{h, p}$ such that
$$
\mathcal{B}(u_{h, p}, v) = (f, v),\;\;\forall\;v\in V_{h, p}
$$
where
$$
\mathcal{B}(u_{h, p}, v) := \int_\Gamma \rho(\xi)\int_D a_m(x, \xi) \nabla_x u_{h, p}(x, \xi)\cdot\nabla_x v(x, \xi) \mathrm{d}x \mathrm{d}\xi,
$$
and
$$
(f, v) := \int_\Gamma \rho(\xi)\int_D v(x, \xi) f(x) \mathrm{d}x\mathrm{d}\xi.
$$
Since $\mathcal{B}(\cdot, \cdot)$ is a symmetric and positive definite bilinear form, it also introduces an inner product and the associated norm denoted by $(\cdot, \cdot)_a$ and $\|\cdot\|_a$, respectively.

We refer to \cite{babuska-tempone-zouraris} for a more thorough discussion of the stochastic Galerkin method.

\section{Matrix structures and iterative solvers}
\label{sec: matrix structures}

In order to design efficient and robust iterative solvers, it is important to study the structure of the corresponding matrix. For example, the Galerkin matrix $A$ is symmetric and positive definite (SPD) by the uniform elliptic assumption (\ref{assumption}). As a consequence, the CG method can be applied to solve the linear system. Since the solution space $V$ and the approximation space $V_h$ are both tensor spaces, the matrix $A$ also contains a tensor product structure. Moreover, $A$ has block sparsity structure and hierarchical structure as described below.

\subsection{Tensor product structure}
Consider the semi-discretization in stochastic domain, the corresponding Galerkin projection $u^{(p)}(\cdot, \xi): \Gamma \rightarrow X_h$ satisfies for each stochastic basis polynomial $\psi_i(\xi)$, $i=1, \dots, N_\xi$:
\begin{equation}
\int_\Gamma 
-\nabla\cdot (a_m(x, \xi) \nabla u^{(p)}(x, \xi)) \, \psi_i(\xi)
\rho(\xi)\,\mathrm{d} \xi
= \int_\Gamma 
f(x) \,\psi_i(\xi)\rho(\xi)\,\mathrm{d} \xi,
\label{semi-discretization}
\end{equation}
where the semi-discrete approximation $u^{(p)} (x, \xi) = \sum_{j=1}^{N_\xi} u_j(x) \psi_j(\xi)$. Note Eqn. (\ref{semi-discretization}) is a system of $N_\xi$ equations with $N_\xi$ unknown functions $\{u_j(x)\}_{j=1}^{N_\xi}$, i.e. the stochastic `stiffness' matrix is given by
\begin{equation*}
\mathcal{A} = \left(
\begin{array}{cccc}
\mathcal{A}_{1, 1} & \mathcal{A}_{1, 2} & \cdots & \mathcal{A}_{1, N_\xi}\\
\mathcal{A}_{2, 1} & \mathcal{A}_{2, 2} & \cdots & \mathcal{A}_{2, N_\xi}\\
\vdots    & \vdots    & \ddots & \vdots\\
\mathcal{A}_{N_\xi, 1} & \mathcal{A}_{N_\xi, 2} & \cdots & \mathcal{A}_{N_\xi, N_\xi}
\end{array}
\right),
\end{equation*}
where each entry $\mathcal{A}_{i, j}$ contains spatial differentiation and is given by
\begin{equation*}
\mathcal{A}_{i, j} = \int_\Gamma -\nabla\cdot (a_m(x, \xi)\nabla u_j(x))\psi_j(\xi)\psi_i(\xi)\rho(\xi)\,\mathrm{d}\xi.
\end{equation*}
Substitute in the truncated Karhunen-Lo\`{e}ve expansion (\ref{truncated-KLE}) for $a_m$, we get
\begin{eqnarray*}
&\mathcal{A}_{i, j}  =  -\nabla\cdot (\bar{a}(x)\nabla u_j(x)) \int_\Gamma \psi_j(\xi)\psi_i(\xi)\rho(\xi)\,\mathrm{d}\xi\\
& -\sum_{k=1}^m \sqrt{\lambda_k}\nabla \cdot (b_k(x) \nabla u_j(x)) \int_\Gamma \xi_k\psi_j(\xi)\psi_i(\xi)\rho(\xi)\,\mathrm{d}\xi.
\end{eqnarray*}

Next, we apply Galerkin projection to discretize spatial differential operators. Let 
$$
u_j(x)= \sum_{s=1}^{N_x} U_{j, s}\phi_s(x),
$$ 
where $\{\phi_s(x)\}_{s=1}^{N_x}$ is the set of finite element basis functions. Multiply the spatial derivative terms in each $\mathcal{A}_{i, j}$ by the test function $\phi_r(x)$ (for $r= 1, \dots, N_x$) and integration by parts gives
\begin{equation}
\sum_{s=1}^{N_x} U_{j, s}\int_D \bar{a}(x)\nabla \phi_s(x) \cdot \nabla \phi_r(x)\,\mathrm{d}x 
,\;\; \sum_{k=1}^m \sqrt{\lambda_k}  \sum_{s=1}^{N_x} U_{j, s} \int_D b_k(x)\nabla\phi_s(x)\cdot \nabla\phi_r(x)\,\mathrm{d}x.
\label{spatial-discretization}
\end{equation}

From (\ref{spatial-discretization}), we define spatial stiffness matrices $K_0$ and $K_k$ (for $k=1, \dots, m$) as
\begin{align}
K_0(r, s) &= \int_D \bar{a}(x)\nabla \phi_s(x)\cdot \nabla \phi_r(x) \,\mathrm{d}x,\nonumber\\
K_k(r, s) &= \sqrt{\lambda_k} \int_D b_k(x)\nabla \phi_s(x)\cdot \nabla \phi_r(x) \,\mathrm{d}x,\;r, s = 1, \dots, N_x.
\label{def-stiffness}
\end{align}
Similarly, we define the stochastic matrices $G_0$ and $\{G_k\}_{k=1}^m$ as
\begin{align}
G_0(i, j) & = \int_\Gamma \psi_j(\xi) \psi_i(\xi)\rho(\xi)\,\mathrm{d}\xi,\nonumber\\
G_k(i, j) & = \int_\Gamma \xi_k \psi_j(\xi) \psi_i(\xi)\rho(\xi)\,\mathrm{d}\xi,\;i, j = 1, \dots, N_\xi.
\label{def-stochastic}
\end{align}
Finally, the Galerkin matrix $A$ can be written in terms of Kronecker products:
\begin{equation}
A = G_0\otimes K_0 + \sum_{k=1}^m G_k \otimes K_k.
\label{tensor-product}
\end{equation}
In practice, one does not assemble the Galerkin matrix $A$ explicitly. Instead, using the tensor product structure (\ref{tensor-product}), only $m+1$ spatial stiffness matrices $K_k$ and $m+1$ stochastic matrices $G_k$ need to be stored.

\subsection{Block sparsity structure}
By construction, the gPC basis functions $\{\psi_i\}_{i=1}^{N_\xi}$ are orthonormal with respect to the PDF $\rho$, i.e.\begin{equation*}
\int_\Gamma \psi_j(\xi)\psi_i(\xi)\rho(\xi)\,\mathrm{d}\xi =  \delta_{i,j}.
\end{equation*}
Furthermore, the stochastic matrices $G_k$ ($k=0, 1, \dots, m$) are sparse as can be seen from the explicit formulas for the matrix elements (\ref{def-stochastic}) given in \cite{rosseel-vandewalle} (Theorem 2.2). As a consequence, the Galerkin matrix $A$ has a particular block sparsity structure, see Fig.\ref{matrix_structure} for the case when $p=4$, $m=4$. This block sparsity structure is essential for the construction of many efficient iterative solvers for stochastic Galerkin methods as seen in \cite{pellissetti-ghanem, powell-elman}.
\begin{figure}
	\centering
      \begin{tabular}{cc}
         \includegraphics[width=65mm]{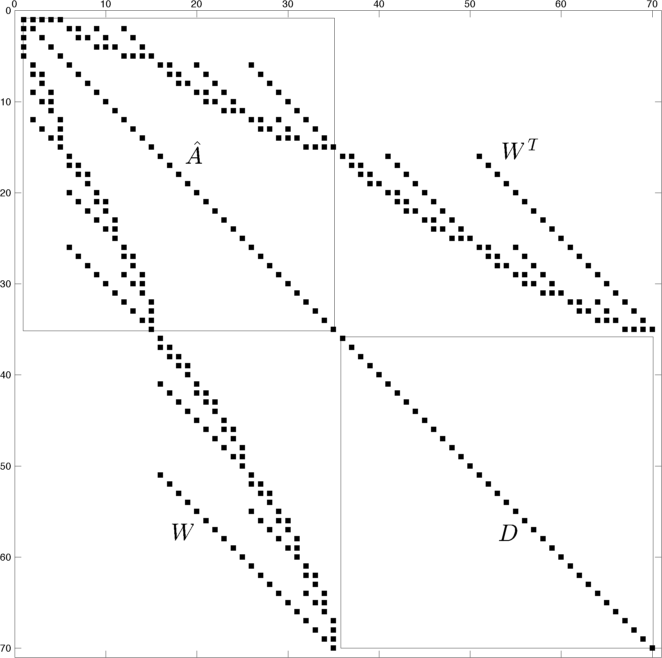}
      \end{tabular}
   \caption{Block sparsity structure of the Galerkin matrix $A$ ($m=4$, $p=4$).}
   \label{matrix_structure}
\end{figure}

\subsection{Hierarchical structure}

In addition to the block sparsity property, the matrix $A$ also has a hierarchical structure due to the hierarchical nature of the stochastic approximation spaces $\Xi_p$, i.e., $\Xi_{p-1}\subset \Xi_{p}$, and the use of gPC basis functions which are also hierarchical. Namely, the basis functions of higher order stochastic approximation space are obtained by adding homogeneous high order polynomials while keeping basis functions from lower order space. We remark that the hierarchical basis functions are also used in many finite element computations for deterministic problems, see \cite{bank-hierarchical-basis} and the references therein.

Let $A$ be the Galerkin matrix corresponding to $X_h \otimes \Xi_{p}$ with gPC basis of order $p$. It can be rewritten as
\begin{equation}
A = \left[
\begin{array}{cc}
\hat{A} & W^T\\
W        & D
\end{array}
\right], \;\; A_0 = K_0,
\label{hierarchical_structure}
\end{equation}
where $\hat{A}$ is the Galerkin matrix corresponding to $X_h\otimes\Xi_{p-1}$, $W$ represents the coupling of higher-order and lower-order stochastic modes, and $D$ is a block diagonal matrix corresponding to the homogeneous gPC basis of degree $p$. 

This hierarchical structure is shown in Fig.\ref{matrix_structure}. In \cite{ghanem_kruger, pellissetti-ghanem}, several hierarchical approaches have been discussed assuming weak coupling of different orders of approximation. Multilevel methods in the stochastic domain based on this hierarchical structure have been designed and used as preconditioners for Krylov methods in \cite{rosseel-vandewalle} which exhibited good convergence properties.

\section{Block preconditioners}
\label{sec: block preconditioners}

Preconditioning techniques are necessary in order to accelerate the convergence of the Krylov subspace methods when applied to ill-conditioned linear systems. Since the Galerkin matrix $A$ has a block sparse structure, it is natural to consider block preconditioners. 

\subsection{Block-diagonal preconditioner}
It is known that the block-diagonal preconditioner (also known as the mean-based preconditioner) defined by
\begin{equation}
B_D := G_0\otimes K_0,
\end{equation}
works very well with the CG method when the variance of the diffusion coefficient $a(x, \omega)$ is small \cite{ghanem_kruger, pellissetti-ghanem}. Spectral bounds for the block-diagonal preconditioned system matrix have been derived in \cite{powell-elman}.

By the uniform ellipticity assumption, the mean stiffness matrix $K_0$ is SPD. Multigrid method can be used to invert each diagonal block approximately.

\subsection{Block-triangular preconditioner}
Another choice is to use the lower block triangular part of the matrix $A_{p}$ as a preconditioner. This can be viewed as applying one step of the block Gauss-Seidel method with zero initial guess. Consider a splitting of the stochastic approximation space $\Xi_p$, 
$$
\Xi_p = \Xi_{p-1}\oplus (\Xi_p\backslash \Xi_{p-1})
$$
and the corresponding splitting of the global approximation space 
$$
V_{h, p} = (X_h\otimes \Xi_{p-1}) \oplus (X_h\otimes (\Xi_p \backslash \Xi_{p-1})),
$$
which results in a $2\times 2$ block structure given by (\ref{hierarchical_structure}). 

Given $u^{(0)}=0$, for $k=1, 2, \dots$, the block Gauss-Seidel iterate $u^{(k+1)}$ is given by the following two steps:
\begin{itemize}
\item
Find $u^{(k+1/2)}\in u^{(k)}+X_h\otimes \Xi_{p-1}$ such that
\begin{align*}
\mathcal{B}(u^{(k+1/2)}, v)
= \int_\Gamma \rho(\xi) \int_D v(x, \xi)
f(x) \;\mathrm{d} x \,\mathrm{d} \xi,\;\;\forall v\in X_h\otimes\Xi_{p-1}.
\end{align*}
\item Find $u^{(k+1)}\in u^{(k+1/2)}+\Xi_h\otimes (\Xi_p\backslash \Xi_{p-1})$
such that
\begin{align*}
\mathcal{B}(u^{(k)}, v) =
\int_\Gamma \rho(\xi) \int_D v(x, \xi)
f(x) \,\mathrm{d} x \,\mathrm{d} \xi,\;\;
\forall v\in X_h\otimes (\Xi_p\backslash \Xi_{p-1}).
\end{align*}
\end{itemize}
In matrix notation, the above block Gauss-Seidel method can be described by the following matrix splitting
\begin{equation*}
A = \left[
\begin{array}{cc}
\hat{A} & 0\\
W   & D
\end{array}
\right]-
 \left[
\begin{array}{cc}
0 & -W^T\\
0 & 0
\end{array}
\right].
\end{equation*}

We define the block triangular preconditioner
\begin{equation}
B_T : = \left[
\begin{array}{cc}
\hat{A} & 0\\
W   & D
\end{array}
\right].
\end{equation}
The corresponding preconditioner system
\begin{equation*}
\left[
\begin{array}{cc}
\hat{A} & 0\\
W & D
\end{array}
\right]\;\left[
\begin{array}{c}
\tilde{U}_1\\
\tilde{U}_2
\end{array}
\right]
=
\left[
\begin{array}{c}
\tilde{F}_1\\
\tilde{F}_2
\end{array}
\right]
\label{triangular-linear-system}
\end{equation*}
may be solved inexactly by the standard multigrid V-cycle.

\begin{remark}
The block triangular preconditioner $B_T$ may also be motivated by considering the block LU factorization
\begin{equation}
A = \left[
\begin{array}{cc}
\hat{A} & 0\\
W   & S_D
\end{array}
\right]\;\left[
\begin{array}{cc}
I & \hat{A}^{-1}W^T\\
0 & I
\end{array}
\right],\;\;S_D = D - W \hat{A}^{-1} W^T.
\label{block-LU}
\end{equation}
It is known that with the ``ideal" block triangular preconditioner
\begin{equation*}
\tilde{B}_T= \left[
\begin{array}{cc}
\hat{A} & 0\\
W   & S_D
\end{array}
\right]
\end{equation*}
the GMRes method converges in at most two iterations. However $\tilde{B}_T$ is impractical because the Schur complement $S_D$ is computationally expensive to invert. Replacing $S_D$ by $D$ in $\tilde{B}_T$ results in the block triangular preconditioner $B_T$.
\end{remark}

\begin{remark}
A different form of the block LU factorization
\begin{equation}
A = \left[
\begin{array}{cc}
I & W^T D^{-1}\\
0 & I
\end{array}
\right]
\;
 \left[
\begin{array}{cc}
S_A & 0\\
0   & D
\end{array}
\right]\;\left[
\begin{array}{cc}
I & 0\\
D^{-1}W & I
\end{array}
\right],\;\;S_A = \hat{A} - W^T D^{-1} W,
\label{block-LU2}
\end{equation}
is studied in \cite{sousedik-ghanem-phipps} from which a hierarchical Schur complement preconditioner is derived. 
\end{remark}

Since $B_T$ is nonsymmetric, we can use it with the GMRes method \cite{saad-schultz} or GPCG method \cite{blaheta}. To apply the standard PCG method, we may consider the block symmetric Gauss-Seidel method as the preconditioner, i.e.
$$
B_{S}:=\left[
\begin{array}{cc}
\hat{A} & 0\\
W & D
\end{array}
\right] \;\left[
\begin{array}{cc}
\hat{A} & 0\\
0 & D
\end{array}
\right]^{-1}
\;\left[
\begin{array}{cc}
\hat{A} & W^T\\
0 & D
\end{array}
\right].
$$
It is clear that $B_S$ is SPD and the standard PCG method is guaranteed to converge. 

\section{Convergence Analysis}
\label{sec: convergence analysis}

In this section, we give eigenvalue bounds for the matrix preconditioned by block triangular preconditioner which is crucial to the convergence of the preconditioned Krylov subspace methods. We note that the spectral properties of block triangular preconditioner for saddle point problems have been studied in \cite{klaxon, simoncini, axelson-blaheta}. 

The following two lemmas are useful.

\begin{lemma} \cite{simoncini}
The eigenvalues of $AB_T^{-1}$ are positive real numbers, and the spectrum
satisfies
$$
\sigma(AB_T^{-1})\subset \{1\}\cup \sigma(S, D)
$$
where $S=D-W \hat{A}^{-1}W^T$ is the Schur complement of $\hat{A}$ in $A$, and $\sigma(S, D)$ contains the eigenvalues $\mu$ corresponding to the generalized
eigenvalue problem
\begin{equation}
S z = \mu D z.
\label{generalized-eigenproblem}
\end{equation}
\label{lemma-spectrum}
\end{lemma}

\begin{lemma} \cite{powell-elman}
Let $K_0$ and $K_k$ be the stiffness matrices defined in (\ref{def-stiffness}), then the eigenvalues of $K_0^{-1}K_k$ belong to the interval
\begin{equation}
\left[
\frac{1}{\bar{a}}\sqrt{\lambda_k}b_k^{\min}, \;\frac{1}{\bar{a}}\sqrt{\lambda_k}b_k^{\max}
\right]
\;\;\text{or}\;\;
\left[
-\frac{1}{\bar{a}}\sqrt{\lambda_k}\|b_k\|_\infty, \;\frac{1}{\bar{a}}\sqrt{\lambda_k}\|b_k\|_{\infty}
\right]
\label{k_0k_k}
\end{equation}
depending on the positivity of $b_k(x)$.
\label{lemma-k0-kk}
\end{lemma}

Next, we give the main result of this section.

\begin{theorem}
The spectrum of the preconditioned Galerkin matrix $AB_T^{-1}$ satisfies
\begin{equation*}
\sigma(AB_T^{-1})\subset (0, 1].
\end{equation*}
\end{theorem}
\begin{proof}
By Lemma \ref{lemma-spectrum}, it suffices to show that $\sigma(S, D)\subset (0, 1]$.

Let $\mu$ be an eigenvalue of the generalized eigenvalue problem (\ref{generalized-eigenproblem}), and $z$ be the corresponding eigenvector,
\begin{align}
&(D-W \hat{A}^{-1}W^T) z = \mu D z
\label{spd-1}
\\
\Rightarrow \;\;
&z^T (D-W \hat{A}^{-1}W^T) z = \mu z^T D z\nonumber\\
\Rightarrow\;\;
&\mu = \frac{z^T (D - W \hat{A}^{-1}W^T) z}{z^T D z}.\nonumber
\end{align}
Since both $S$ and $D$ are symmetric positive definite, we know $\mu > 0 $.

Similarly, from (\ref{spd-1}), 
\begin{align*}
& (1-\mu)D z = W \hat{A}^{-1}W^T z\nonumber\\
\Rightarrow\;\;
& (1-\mu)z^T D z = (W^T z)^T \hat{A}^{-1}(W^T z)\nonumber\\
\Rightarrow\;\;
& 1-\mu = \frac{(W^T z)^T \hat{A}^{-1}(W^T z)}{z^T D z},
\end{align*}
and the symmetric positive definiteness of $D$ and $\hat{A}^{-1}$, we conclude that
$$
1-\mu \geq 0\;\;\Rightarrow \;\;\mu \leq 1.
$$
\end{proof}

For the case $p=1$, the following result gives a better lower bound for $\mu$.

\begin{theorem}
When $p=1$, the eigenvalue $\mu$ of the preconditioned Galerkin matrix $A B_T^{-1}$ satisfies
$$
\mu > 1-  c\sum_{k=1}^m \frac{1}{\bar{a}^2}\lambda_k\|b_k\|^2_{\infty}.
$$
\end{theorem}

\begin{proof} Notice $\sigma(A B_T^{-1}) = \sigma(B_T^{-1}A)$, and
\begin{align*}
B_T^{-1}A =
\left[
\begin{array}{cc}
\hat{A}  & 0\\
W & D
\end{array}
\right]^{-1}
\left[
\begin{array}{cc}
\hat{A} & W^T\\
W & D
\end{array}
\right]
= I + 
\left[
\begin{array}{cc}
\hat{A}  & 0\\
W & D
\end{array}
\right]^{-1}
\left[
\begin{array}{cc}
0 & W^T\\
0 & 0
\end{array}
\right]
\end{align*}
Let $\mu\in \sigma(B_T^{-1} A)$, $v=[v_1; v_2]$ be the corresponding eigenvector. We have
$$
 \left[
\begin{array}{cc}
0 & W^T\\
0 & 0
\end{array}
\right]
\left[
\begin{array}{c}
v_1\\
v_2
\end{array}
\right]
=(\mu-1)\left[
\begin{array}{cc}
\hat{A} & 0\\
W & D
\end{array}
\right]
\left[
\begin{array}{c}
v_1\\
v_2
\end{array}
\right]
$$
which implies
$$
(\hat{A}-W^T D^{-1} W) v_1 = \mu  \hat{A} v_1.
$$
Hence,
$$
(I-\hat{A}^{-1}W^T D^{-1}W)v_1 = \mu v_1.
$$
Notice that
$$
\hat{A} = K_0,\;\;W = \sum_{k=1}^m \tilde{G}_k \otimes K_k, \;\;D = \tilde{G}_0\otimes K_0
$$
where $K_0$, $K_k$ are defined by (\ref{def-stiffness}), $\tilde{G}_k$ is the first column of $G_k$ defined by (\ref{def-stochastic}), and $\tilde{G}_0 = I_{m\times m}$ assuming the stochastic basis functions are normalized.

Using the properties of the Kronecker product, we get
$$
D^{-1} W = \sum_{k=1}^m \tilde{G}_k\otimes K_0^{-1}K_k,\;\;
\hat{A}^{-1} W^T = \sum_{l=1}^m \tilde{G}_l^T \otimes K_0^{-1}K_k,
$$
and
\begin{align*}
\hat{A}^{-1}W^T D^{-1}W &= \sum_{l=1}^m \left[
(\tilde{G}_l^T\otimes K_0^{-1}K_l)
\left(\sum_{k=1}^m \tilde{G}_k\otimes K_0^{-1}K_k
\right)
\right]\\
& = \sum_{l=1}^m \sum_{k=1}^m \left[
(\tilde{G}_l^T\tilde{G}_k)\otimes (K_0^{-1}K_l)(K_0^{-1}K_k)
\right]\\
& = \sum_{k=1}^m (\tilde{G}_k^T\tilde{G}_k)\otimes (K_0^{-1}K_k)^2\\
& = c \sum_{k=1}^m  (K_0^{-1}K_k)^2,
\end{align*}
where $c$ is a constant depending on the stochastic basis functions.

Applying Lemma \ref{lemma-k0-kk} to conclude that
$$
\mu \geq 1-c\sum_{k=1}^m \frac{1}{\bar{a}^2}\lambda_k\|b_k\|^2_{\infty}.
$$
\end{proof}

For $p>1$, we do not have a similar estimate for the lower bound of $\mu$ due to the lack of simple formulae for $\hat{A}^{-1}$ which is the inverse of a sum of Kronecker products. Furthermore, it is well known that the eigenvalue information alone is not sufficient to predict the convergence behavior of the GMRes method \cite{campbell-ipsen-kelley-meyer}. On the other hand, we observe the uniform convergence rate with respect to both mesh size $h$ and the stochastic discretization parameters $p$ and $m$ in the numerical experiments described in Section \ref{sec: numerical results}.

\section{Numerical results}
\label{sec: numerical results}

In this section, we evaluate the performance of the block triangular preconditioner and compare with the block diagonal preconditioner. The test problem is taken from \cite{deb-babuska-oden, powell-elman}. Let $D = (-0.5, 0.5) \times (-0.5, 0.5)$. The covariance function is given by ($\ref{exponential-covariance}$) where $\sigma = 0.1$, $0.2$, $0.3$, or $0.4$ and $L=1$. The independent random variables $\xi_i$ are assumed to be uniformly distributed with range $\gamma_i = (-\sqrt{3}, \sqrt{3})$, for $i = 1, \dots, m$. The mean value of the random coefficient $\bar{a}(x) = 1$, and the right hand side $f(x, y)=2(0.5-x^2-y^2)$.

In our implementation of the stochastic Galerkin method, the linear finite element method  is used for spatial discretization on the uniform triangulation with mesh size $h$. Multivariate Legendre polynomials of total degree $\leq p$ are used for the stochastic discretization. For each Krylov subspace method, the outer iteration starts from a zero initial guess and ends when the relative residual error in Eucliean norm less than $10^{-10}$ is achieved. All computations are done in MATLAB on a laptop with a 2.8 GHz Intel processor and 4 GB of memory.

In the following tables, we use the abbreviations, B-GS : block Gauss-Seidel method; BD-PCG: PCG with block diagonal preconditioner; BT-GPCG: GPCG[1] method (for details see \cite{blaheta}) with block triangular preconditioner; BT-GMRes: GMRes(10) method (restarted every $10$ iterations if needed) with block triangular preconditioner; BS-PCG: PCG with block symmetric Gauss-Seidel preconditioner.

The sparsity information for the relevant matrices used in the block diagonal preconditioner $B_D$ and the block triangular preconditioner $B_T$ is reported in Table \ref{size-nnz}.

\begin{table}[htbp]
\caption{Size and number of nonzeros of the relevant matrices ($p=4$, $m=6$, $h=1/64$).}\label{size-nnz}
\begin{center}
\begin{tabular}{|c|c|c|c|c|c|}
\hline
 &$A$ & $\hat{A}$ & $W$ & $D$ & $B_D$\\
      \hline
size   & $833, 490$  & $333, 396$ & $(500, 094) \times (333, 396)$ & $500, 094$ & $833, 490$\\
\hline
nnz    & $23, 863, 938$  & $8, 228, 948$ & $6, 583, 136$ & $2, 468, 718$ & $4, 114, 530$\\
\hline
\end{tabular}
\end{center}
\end{table}

In Table \ref{uniform-table}, we report the performance of the block preconditioners with respect to the spatial mesh size $h$, and the stochastic discretization parameters $p$ and $m$ when $\sigma=0.1$ is fixed. From Table \ref{uniform-table}, it is shown that with the block diagonal, block triangular preconditioners, or block symmetric Gauss-Seidel preconditioner, the Krylov subspace methods converge uniformly with respect to $h$, $p$, and $m$. 

\begin{table}[htbp]
\caption{Number of iterations with diagonal blocks solved by one multigrid $V(2, 2)$ cycle ($\sigma = 0.1$, $m=4$ or $6$).}\label{uniform-table}
\begin{center}
\begin{tabular}{|c||c|c|c||c|c|c||c|c|c|}
\hline
      & \multicolumn{3}{|c||}{BD-PCG} &    \multicolumn{3}{|c||}{BT-GPCG / BS-PCG}  & \multicolumn{3}{|c|}{BT-GMRes}  \\
\hline
  h    &  $p=2$ & $p=3$ & $p=4$ & $p=2$ & $p=3$ & $p=4$ & $p=2$ & $p=3$ & $p=4$ \\
\hline
	  $1/32$ & $13$ & $13$ & $13$ & $9$ & $9$ & $9$ & $8$ & $8$ & $8$\\
\hline
	  $1/64$ & $13$ & $13$ & $13$ & $9$ & $9$ & $9$ & $8$ & $8$ & $8$\\
\hline
	  $ 1/128$ & $13$ & $13$ & $13$ & $9$ & $9$ & $9$ & $8$ & $8$ & $8$\\
\hline
\end{tabular}
\end{center}
\end{table}

\begin{figure}
	\centering
      \begin{tabular}{cc}
         \includegraphics[width=75mm]{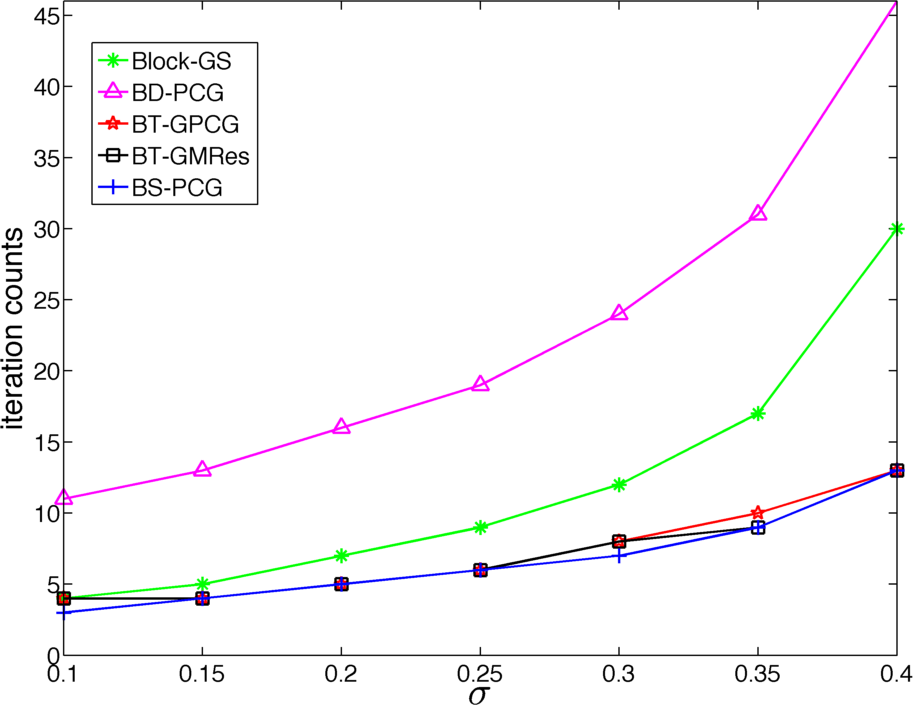}
      \end{tabular}
   \caption{Number of iterations with exact solve of diagonal blocks ($p=4$, $m=6$, $h=1/64$).}
   \label{exact-solve}
\end{figure}


Figure \ref{exact-solve} shows the robustness of the different iterative methods with respect to the diffusivity variance $\sigma$ when the diagonal blocks are solved exactly.
It can be seen that the number of iterations increases when $\sigma$ increases for all the iterative solvers considered. This seems to be related to the fact that large variance indicates strong coupling between the stochastic modes of different order. Moreover, the matrix $A$ may become indefinite for large $\sigma$. The methods using block triangular preconditioner are more robust compared with the block diagonal preconditioned CG method. However, the cost associated with block triangular preconditioner is larger than that of the block diagonal preconditioner. Hence, the gain in terms of computational efficiency (CPU time) is not as significant as seen from iteration counts, see Table \ref{V(2,2)-solve}. 

\begin{table}[htbp]
\caption{Number of iterations and CPU time (in seconds) with one V(2, 2) for diagonal blocks ($p=4$, $m=6$, $h=1/64$).}\label{V(2,2)-solve}
\begin{center}
\begin{tabular}{|c|c|c|c|c|}
\hline
  $\sigma$    &  $0.1$ & $0.2$ & $0.3$ & $0.4$\\
  \hline
	  Block-GS  & $13 (15.53) $ & $16 (18.87) $ & $24 (30.67) $ & $65 (76.35) $ \\
\hline
	  BD-PCG  & $13 (12.18) $ & $18 (15.79) $ & $27 (24.13) $ & $49 (43.40) $ \\
\hline
	  BT-GPCG  & $9 (12.91) $ & $10 (13.42)$ & $13 (17.53) $ & $22 (30.15) $\\
\hline
	  BT-GMRes  & $8 (16.18) $ & $9 (17.01) $ & $12 (23.24) $ & $20 (37.88) $\\
\hline
	  BS-PCG & $9 (19.12)$ & $10 (20.96) $ & $12 (27.17) $ & $ 20 (44.24) $\\
\hline
\end{tabular}
\end{center}
\end{table}

Table \ref{V(2,2)-solve} shows the iteration counts and CPU time (in seconds) comparison of the five iterative solvers. Note that here all the diagonal blocks are solved inexactly by applying a single geometric multigrid V-cycle with two pre- and two post-smoothing (point Gauss-Seidel smoother). From this table we can see that block triangular preconditioner performs better than block diagonal preconditioner when $\sigma$ is large. We also point out that when $\sigma$ is small, the block Gauss-Seidel method also gives fairly good results comparable to those from the preconditioned Krylov subspace methods. The block symmetric Gauss-Seidel preconditioned CG method is computationally more expensive than other methods although its number of iterations is small.

\section{Conclusions}
In this work we study block preconditioners for the coupled linear systems resulting from the stochastic Galerkin discretizations of the elliptic stochastic problem. The proposed block triangular preconditioner utilizes the block sparsity and hierarchical structure of the stochastic Galerkin matrix. The preconditioner is solved inexactly by geometric multigrid V-cycle and applied to Krylov subspace methods such as GMRes or GPCG. A symmetrized version of the preconditioner is proposed and applied to the standard PCG method. Numerical results indicate that the block triangular preconditioner achieves better efficiency compared to the traditional block diagonal preconditioner especially for problems with large variance. We also give theoretical bounds for the spectrum of the preconditioned system. 

\bibliographystyle{plain}
\bibliography{spdemg}

\end{document}